\documentclass[
11pt
]{amsart}
\usepackage{amsmath,amsfonts,amsthm,amssymb,dsfont}
\usepackage[alphabetic]{amsrefs}
\usepackage[OT4]{fontenc}
\usepackage{enumerate}
\usepackage{enumerate, xspace}
\usepackage{graphicx}

\usepackage[small,nohug,heads=vee]{diagrams}
\diagramstyle[labelstyle=\scriptstyle]

\usepackage[small]{caption}

\long\def\symbolfootnote[#1]#2{\begingroup%
\def\thefootnote{\fnsymbol{footnote}}\footnote[#1]{#2}\endgroup}

\newtheorem{theorem}{Theorem}[section]
\newtheorem{lemma}[theorem]{Lemma}
\newtheorem{lem}[theorem]{Lemma}

\newtheorem{proposition}[theorem]{Proposition}
\newtheorem{constr}[theorem]{Construction}
\newtheorem{cor}[theorem]{Corollary}
\newtheorem{step}{Step}

\theoremstyle{definition}

\newtheorem{rem}[theorem]{Remark}
\newtheorem{defin}[theorem]{Definition}

\newcommand{\R}{\mathbf{R}}
\newcommand{\Z}{\mathbf{Z}}

\begin{document}

\title{Graph manifolds with boundary are virtually special}
\author[P.~Przytycki]{Piotr Przytycki$^\dag$}
\address{Inst. of Math., Polish Academy of Sciences\\
 \'Sniadeckich 8, 00-956 Warsaw, Poland}
\email{pprzytyc@mimuw.edu.pl}
\thanks{$\dag$ Partially supported by MNiSW grant N201 012 32/0718 and the Foundation for Polish Science.}
\author[D.~T.~Wise]{Daniel T. Wise$^\ddag$}
           \address{Math. \& Stats.\\
                    McGill University \\
                    Montreal, Quebec, Canada H3A 2K6 }
           \email{wise@math.mcgill.ca}
\thanks{$\ddag$ Supported by NSERC}


\maketitle

\begin{abstract}
\noindent
Let $M$ be a graph manifold. We prove that fundamental groups of embedded incompressible surfaces in $M$ are separable in $\pi_1M$, and that the double cosets for crossing surfaces are also separable. We deduce that if there is a ``sufficient'' collection of surfaces in $M$, then $\pi_1M$ is virtually the fundamental group of a special nonpositively curved cube complex.
We provide a sufficient collection for graph manifolds with boundary thus proving that their fundamental groups are virtually special, and hence linear.
\end{abstract}

\section{Introduction}

A \emph{graph manifold} is an oriented compact connected $3$--manifold that is irreducible and has only Seifert-fibered pieces in its JSJ decomposition. Hempel proved that the fundamental groups of all Haken $3$--manifolds, in particular all graph manifolds, are residually finite \cite[Thm~1.1]{He}. Throughout the article we assume that a graph manifold is not a single Seifert-fibered space and not a Sol manifold. For background on graph manifolds we refer to the survey article by Buyalo and Svetlov \cite{BS}.

We are interested in separability properties of surfaces properly embedded in graph manifolds. A subgroup $F$ of a group $G$ is \emph{separable} if for each $g\in G- F$, there is a finite index subgroup $H$ of $G$ with $g\notin HF$. Let $S$ be an oriented incompressible surface embedded in a graph manifold $M$. Then $\pi_1S$ embeds in $\pi_1M$ by the loop theorem.

\begin{theorem}
\label{thm:separability}
Let $M$ be a graph manifold (with or without boundary). Let $S$ be an oriented incompressible surface embedded in $M$. Then $\pi_1S$ is separable in $\pi_1M$.
\end{theorem}

More generally, consider subgroups $F_1,F_2\subset G$. The double coset $F_1F_2$ is \emph{separable} if for each $g\in G- F_1F_2$ there is a finite index subgroup $H$ of $G$ with $g\notin HF_1F_2$.

We identify the group of covering transformations of the universal cover $\widetilde{M}$ of $M$ with $\pi_1M$.

\begin{theorem}
\label{thm:double_separability}
Let $M$ be a graph manifold. Let $S,P\subset M$ be oriented incompressible surfaces whose intersections with each block are horizontal or vertical (see Section~\ref{sec:not}). Let $\widetilde{S},\widetilde{P}\subset \widetilde{M}$ be intersecting components of the preimages of $S,P$.
Then $\mathrm{Stab}(\widetilde{S}) \mathrm{Stab}(\widetilde{P})$ is separable in $\pi_1M$.
\end{theorem}

We apply Theorems~\ref{thm:separability}~and~\ref{thm:double_separability} to obtain:
\begin{cor}
\label{cor:virtually special}
Let $M$ be a graph manifold with non-empty boundary. Then $\pi_1M$ is virtually the fundamental group of a special cube complex.
\end{cor}

A \emph{special} cube complex is a nonpositively curved cube complex that admits a local isometry into the Salvetti complex of a right-angled Artin group (see \cite{HW} and \cite{HW2}). As a consequence, the fundamental groups of special cube complexes (which are also called \emph{special}) are subgroups of right-angled Artin groups. The latter have various outstanding properties. To mention just a few, they are linear \cite{Hump} and residually torsion-free nilpotent \cite{DK}. Moreover, they virtually satisfy Agol's RFRS condition \cite{A}.

It was proven that fundamental groups of closed hyperbolic $3$--manifolds are virtually special \cite{Hier,Ahak}. A relative version of this theorem says that the fundamental groups of hyperbolic $3$--manifolds with boundary are virtually special as well \cite{Hier}. Our theorem treats the complementary case, with an eye towards eventually analyzing the case of manifolds with both hyperbolic and Seifert-fibered pieces \cite{PWmixed}.

The class of graph manifolds with boundary has been studied by Wang and Yu who prove \cite[Theorem 0.1]{WY} that they all virtually fiber over the circle. (Note that we do not exploit that result in our article.) A closed graph manifold might not virtually fiber \cite{LW}. Hence, by Agol's virtual fibering criterion \cite{A} such a manifold cannot have a virtually special fundamental group. Thus some restriction is needed in Corollary~\ref{cor:virtually special}.

In fact, we have recently learned that independently Yi Liu has proved \cite[Thm~1.1]{Liu} that the graph manifolds with virtually special fundamental groups are exactly the ones that admit a nonpositively curved Riemannian metric. It was proved by Leeb \cite[Thm~ 3.2]{Leeb} that graph manifolds with boundary admit a nonpositively curved Riemannian metric (with geodesic boundary). Hence our Corollary~\ref{cor:virtually special} is a special case of the theorem of Liu.
\smallskip

In order to obtain Corollary~\ref{cor:virtually special} we prove, using Theorems~\ref{thm:separability} and~\ref{thm:double_separability}, the following criterion involving ``sufficient'' collections of surfaces. (For definitions see Section~\ref{sec:not}.)

\begin{defin}Let $\mathcal{S}$ be a collection of incompressible oriented surfaces embedded in a graph manifold $M$ satisfying the property that the intersection of each surface from $\mathcal{S}$ with each block of $M$ is vertical or horizontal. We say that $\mathcal{S}$ is \emph{sufficient} if:
\begin{enumerate}[(1)]
\item
for each block $B\subset M$ and each torus $T\subset \partial B$, there is a surface $S\in \mathcal{S}$ such that $S\cap T$ is non-empty and vertical w.r.t.~$B$,
\item
for each block $B\subset M$ there is a surface $S\in \mathcal{S}$  such that $S\cap B$ is horizontal.
\end{enumerate}
\end{defin}

Note that property (1) automatically implies property (2). Indeed, let $B_0$ be a block and let $B_1$ be any adjacent block. Let $T=B_0\cap B_1$. By (1) there is a surface $S$ such that $S\cap T$ is vertical in $B_1$. Then $S\cap B_0$ is horizontal.

Our criterion is the following:

\begin{theorem}
\label{thm:suff implie special}
Assume that a graph manifold $M$ admits a sufficient collection $\mathcal{S}$. Then $\pi_1M$ is virtually special.
\end{theorem}

Once we prove Theorem~\ref{thm:suff implie special}, in order to derive Corollary~\ref{cor:virtually special} it remains to construct a sufficient collection for graph manifolds with boundary.

In the course of his proof \cite[Lem~4.7]{Liu} Liu constructs a set of cohomology classes giving rise to a sufficient collection. Hence combining this with Theorem~\ref{thm:suff implie special} one can get an alternate argument for Liu's theorem that all graph manifolds admitting a nonpositively curved Riemannian metric have virtually special fundamental groups. Liu suggests to us that cut-and-paste operations on the surfaces obtained in \cite[\S5.5.3]{BS} also yield a sufficient collection.

\smallskip

The article is organized as follows: In Section~\ref{sec:not} we discuss notation. In Section~\ref{sec:cub} we derive Corollary~\ref{cor:virtually special} from Theorems~\ref{thm:separability} and \ref{thm:double_separability}. More precisely, we first prove Theorem~\ref{thm:suff implie special} and then prove that graph manifolds with boundary virtually have a sufficient collection (Proposition~\ref{prop:suff exists}). In Section~\ref{sec:block_separation} we prepare the background for the proofs of Theorem~\ref{thm:separability} in Section~\ref{sec:separability} and Theorem~\ref{thm:double_separability} in Section~\ref{sec:double}.

\subsection*{Acknowledgements} We thank Yi Liu for providing feedback on our preprint leading to many improvements.

\section{Notation}
\label{sec:not}

A graph manifold will be denoted by $M$. 
The JSJ tori decompose $M$ into pieces called \emph{blocks} (denoted usually by $M_v$ or $B$).
By passing to a finite degree cover \cite[Prop~4.4]{LW} we can assume that all the blocks are
products $M_v=S^1\times F_v$, where $F_v$ is an oriented surface with at least two boundary components
and nonzero genus. Then $M$ is called \emph{simple}. The induced quotient
map $\pi_1M_v\rightarrow \pi_1F_v$ does not depend on the choice of the product structure.

Let $S$ be an oriented incompressible surface embedded in $M$. An \emph{elevation} $S'\rightarrow M'$ of the embedding $S\rightarrow M$ is an embedding of a cover $S'$ of $S$ into a cover $M'$ of $M$ such that the diagram below commutes. (A \emph{lift} is an elevation with $S'=S$.)

\begin{diagram}
S' &\rTo&M'    \\
\dTo         &    &\dTo   \\
S          &\rTo&M
\end{diagram}

The surface $S$ can be homotoped
so that each component of $S\cap M_v$ (called a \emph{piece}) is either \emph{vertical}
(fibered by the circles of the Seifert fibration and essential) or \emph{horizontal} (transverse to the fibers thus covering
$F_v$). The only exception is when $S$ is a $\partial$--parallel annulus. We discuss this case
separately in Remark~\ref{rem:annulus}.

Since $S$ is embedded, for each block $M_v$ the components of $S\cap M_v$ are either all
horizontal or all vertical, or else $S\cap M_v$ is empty. We accordingly call the block
\emph{$S$--horizontal, $S$--vertical} or \emph{$S$--empty}. When the surface
$S$ in question is understood, we simply call the block \emph{horizontal, vertical} or
$\emph{empty}$.

We shall consider (possibly noncompact) covers $M'\rightarrow M$ of graph manifolds. The connected components in $M'$ of the preimage of blocks of $M$ will also be called \emph{blocks}. When a specified elevation of $S$ crosses a block $M'_v\rightarrow M_v$, then this block will be called \emph{horizontal} or \emph{vertical} if $M_v$ is such. Other blocks of $M'$ will be called \emph{empty}.

\section{Cubulation}
\label{sec:cub}

\begin{proof}[Proof of Theorem~\ref{thm:suff implie special}]
Complete the collection $\mathcal{S}$ to $\mathcal{S}'$ by adding all JSJ tori, and adding a collection of vertical tori in each block $M_v$ whose base curves on the surface $F_v$ \emph{fill} $F_v$. With respect to some hyperbolic metric on $F_v$ this means that the complementary regions of the union of the geodesic representatives of the base curves are discs or annular neighborhoods of the boundary. Note that if we add to that family of curves the base arcs of the annuli guaranteed by property (1) of a sufficient collection, all the complementary regions become discs. Call such a family \emph{strongly filling}.

After a homotopy we can assume that the surfaces in $\mathcal{S}'$ are pairwise transverse. Each elevation of an incompressible surface from $\mathcal{S}'$ to the universal cover $\widetilde{M}$ of $M$ splits $\widetilde{M}$ into two components (up to a set of measure $0$). This gives $\widetilde{M}$ the structure of a ``space with walls'' (see \cite{CN} or \cite{Nica}). We can consider the action of $\pi_1M$ on the associated dual CAT(0) cube complex $X$.

We claim that $\pi_1M$ acts freely on $X$. To justify this, pick $g\in \pi_1M$. If $g$ does not stabilize some block $\widetilde{B}\subset \widetilde{M}$, then it acts freely on the tree that is the underlying graph of the graph manifold structure of $\widetilde{M}$. Hence $g$ also acts freely on $X$, since we have included the JSJ tori in $\mathcal{S}'$. Otherwise, suppose $g$ belongs to the stabilizer of $\widetilde{B}$ identified with $\pi_1B$ for some block $B\subset M$. If $g$ is not central in $\pi_1B$, then by the strong filling property for the vertical pieces within $B$, the element $g$ acts freely on the tree that is dual to the preimage in $\widetilde{B}$ of one of the pieces. Since every elevation of a surface in $\mathcal{S}'$ to $\widetilde{M}$ has connected intersection with $\widetilde{B}$, this implies that
$g$ acts freely on $X$. Otherwise, $g$ is central in $\pi_1B$ and the claim follows from
the existence of a horizontal piece in $B$  among the surfaces in $\mathcal{S}'$ (property (2) of a sufficient collection).
Note that in most cases the action of $\pi_1M$ on $X$ is not cocompact.

We now invoke \cite[Thm 4.1]{HW2}, which is a criterion for $\pi_1M\backslash X$ to be virtually special.
In the case of a cube complex $X$ arising from a collection $\mathcal{S}'$ of compact $\pi_1$--injective surfaces in a $3$--manifold $M$, this criterion is satisfied when:

\begin{enumerate}[(1)]
\item
$\mathcal{S}'$ is finite,
\item
for each surface $S\in \mathcal{S}'$, in the $\pi_1S$ cover $M^S=\pi_1S\backslash \widetilde{M}$ of $M$ there are only finitely many elevations of surfaces in $\mathcal{S}'$ disjoint from $S$, but not separated from $S$ by another elevation of a surface from $\mathcal{S}'$,
\item
for each $S\in\mathcal{S}'$ the subgroup $\pi_1S$ is separable in $\pi_1M$,
\item
for each pair of intersecting elevations $\widetilde{S}, \widetilde{P}\subset \widetilde{M}$ of $S,P\in \mathcal{S}'$, the double coset $\mathrm{Stab}(\widetilde S)\mathrm{Stab}(\widetilde P)$ is separable in $\pi_1M$.
\end{enumerate}

Condition (1) is immediate, conditions (3) and (4) are supplied by Theorems~\ref{thm:separability} and \ref{thm:double_separability}. It remains to discuss condition (2):

Fix $S\in \mathcal{S}'$ and let $P^S$ be an elevation of a surface in $\mathcal{S}'$ to the  $\pi_1S$ cover $M^S$ of $M$. Assume that $P^S$ is disjoint from (the lift of) $S$ but not separated from $S$ by another elevation of a surface from $\mathcal{S}'$. Then $P^S$ must intersect at least one (of the finitely many) blocks $B$ of $M^S$ intersecting $S\subset M^S$ (otherwise an elevation of a JSJ torus separates $S$ and $P^S$). We fix the block $B$. Assume first that $P^S\cap B$ is horizontal and that a component of $P^S\cap B$ projects to a specified piece of the finitely many pieces of $\mathcal{S}'$. Then there can be at most 2 such $P^S$, since they are all nested.

Now assume that $P^S\cap B$ is vertical. Thus $S\cap B$ is also vertical. The entire configuration can then be analyzed using the base curves on the base surface $F$ of $B$. For a strongly filling family of curves, each pair of their elevations to the universal cover of $F$ not separated by a third one has to be at a uniformly bounded distance. Hence there are again only finitely many possible $P^S$. This concludes the argument for condition~(2).

Hence all the conditions of virtual specialness criterion above are satisfied and the cube complex $\pi_1M\backslash X$ is virtually special.
\end{proof}

As an application, we will consider graph manifolds with boundary.

\begin{proposition}
\label{prop:suff exists}
A graph manifold $M$ with non-empty boundary has a finite cover with a sufficient collection.
\end{proposition}

Note that combining Proposition~\ref{prop:suff exists} with Theorem~\ref{thm:suff implie special} yields Corollary~\ref{cor:virtually special}.

In the proof of Proposition~\ref{prop:suff exists} we will need the following:

\begin{lemma}[{version of \cite[Lem~1.1]{WY}}]
\label{lem: finding horizontal}
Let $T_1,\ldots, T_n$ be the boundary components of a block $M_v$. Assume we are given families of disjoint identically oriented circles $C_1\subset T_1,\
\ldots,\ C_{n-1}\subset T_{n-1}$ such that the oriented intersection number between $C_i$ and the vertical fiber is non-zero and independent of $i$. Then there is a family of disjoint identically oriented circles $C_n\subset T_n$ such that $\bigcup_{i=1}^n C_i$ is the boundary of an oriented horizontal surface embedded in $M_v$.
\end{lemma}

\begin{proof}[Proof of Proposition~\ref{prop:suff exists}.]
Let $\Gamma$ be the underlying graph of $M$. A vertex $w$ of $\Gamma$ is called a \emph{boundary vertex} if its block $M_w$ has a torus boundary component contained in $\partial M$. Note that a boundary vertex exists since $\partial M$ is non-empty.

We first pass to a finite cover of $M$ that is simple (see Section~\ref{sec:not}). Moreover, we shall pass to a finite cover whose underlying graph $\Gamma$ has the following property:

\begin{description}
\item[(antennas)]
For each pair of adjacent vertices $v_0, v_1\in \Gamma$ there is an embedded edge-path $(v_0,v_1,v_2,\ldots, v_n)$ such that:
\begin{enumerate}[(i)]
\item
the subpath $(v_1,v_2,\ldots, v_n)$ is a full subgraph (i.e.\ induced subgraph) of $\Gamma$,
\item
$v_n$ is a boundary vertex.
\end{enumerate}
\end{description}

We will first construct a sufficient collection under the assumption that (antennas) property holds. We later explain how to pass to a cover satisfying (antennas).

As discussed in the introduction, it suffices to obtain property (1) of a sufficient collection. Let $B=M_{v_0}$ be a block and let $T$ be a torus in its boundary. Let $C_0$ be the circle on $T$ that is vertical with respect to $B$. If $T$ is a boundary torus of the whole $M$, then we put $n=0$, otherwise we define $v_1$ so that $M_{v_1}$ is the block distinct from $M_{v_0}$ containing $T$. Applying (antennas), we obtain an edge-path satisfying (i) and (ii). We will find a properly embedded surface $S_n$ intersecting $T$ along circles in the direction of $C_0$.

For $i=0$ to $n$ we inductively define surfaces $S_0\subset \cdots \subset S_i\subset \cdots \subset S_n$ embedded in $M$, but not necessarily properly: $S_i$ might have boundary components in $M_{v_i}\cap M_{v_{i+1}}$.

We define the surface $S_0$ to be the vertical annulus in $B=M_{v_0}$ joining $T$ to itself, not separating $M_{v_0}$ (this uses that $M$ is simple). If $n=0$, then we are done.

Otherwise, let $i\geq 1$ and assume that $S_{i-1}$ has already been constructed but is not proper. Let $C_{i-1}$ denote one of the boundary circles of $S_{i-1}$ in $M_{v_{i-1}}\cap M_{v_i}$. If $C_{i-1}$ is vertical in $M_{v_i}$, then we can complete $S_{i-1}$ immediately to $S_n$ by adding several vertical annuli in $M_{v_i}$.

Otherwise, let $\mathcal{E}$ be the family of all edges adjacent to $v_i$, distinct from the edges joining it to $v_{i-1}$ and $v_{i+1}$ (if it is defined). Since by (antennas) property the path $(v_j)_{j=1}^n$ is full, all the edges in $\mathcal{E}$ join $v_i$ to a vertex outside the path $(v_j)_{j=1}^n$. Similarly as we have done for the edge $(v_1,v_0)$, for each edge $e=(v_i,w)\in \mathcal{E}$ we take a vertical annulus $A_e$ in $M_w$ joining the boundary torus of $M_w$ corresponding to $e$ to itself. Again we require that $A_e$ does not separate $M_w$ and because of that we can take it disjoint from all the annuli in $M_w$ constructed for smaller values of $i$, assigned to other boundary components (see Figure~\ref{fig:annuli}).

\begin{figure}[ht]
\begin{center}
\includegraphics[width=4cm]{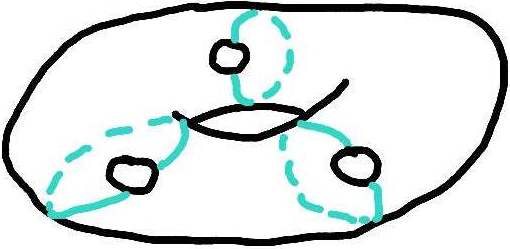}
\end{center}
\caption{bases of $A_e$ in common $M_w$ from different $i$ }
\label{fig:annuli}
\end{figure}

The annuli $A_e$ specify circles $C_e$ on the tori $M_w\cap M_{v_i}$. Thus far, for all but one (or all if $i=n$) boundary tori of $M_{v_i}$ that are not in the boundary of $M$ we have constructed non-vertical circles $C_{i-1}$ or $C_e$. For all the boundary tori $K$ of $M$ in $M_{v_i}$, except for one (call it $Q$) when $i=n$, we pick arbitrary horizontal circles $C_K$.

By Lemma~\ref{lem: finding horizontal}, if we take appropriate orientations on the
circles $C_{i-1},C_e,C_K$ and we take appropriately many copies, we can find an oriented
circle $C_i$ on the remaining boundary torus of $M_{v_i}$ (connecting to $M_{v_{i+1}}$, or
being $Q$), such that appropriately many copies of $C_i$ together with the copies of $C_{i-1},C_e$ and $C_K$ bound an embedded horizontal surface $H_i$.

Taking appropriately many copies of $A_e, S_{i-1}$ and $H_i$ we form the surface $S_i$. If it is non-orientable, we replace it by the boundary of its regular neighborhood.

Inductively, we arrive at the required surface $S_n$ needed for property (1) of a sufficient collection. See Figure~\ref{fig:antenna}.

\begin{figure}[ht]
\begin{center}
\includegraphics[width=7cm]{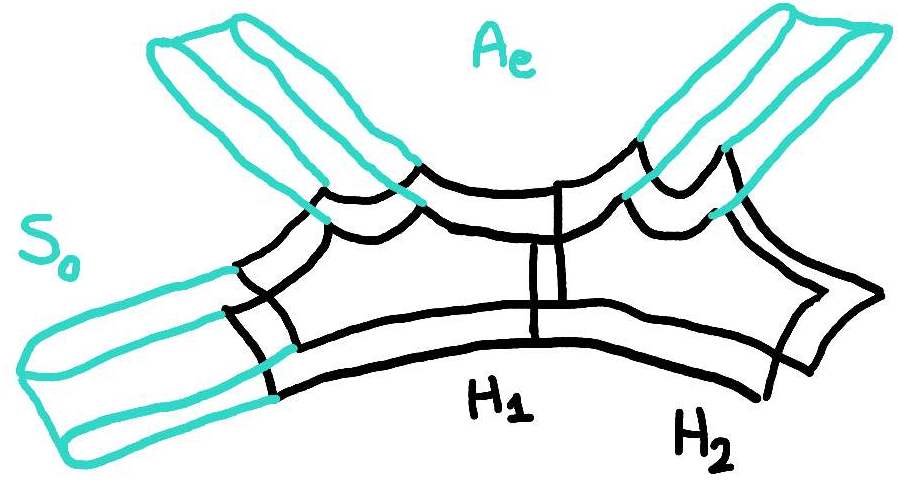}
\end{center}
\caption{the surface $S_2$}
\label{fig:antenna}
\end{figure}

\smallskip

It remains to explain how to obtain property (antennas). Fullness is automatic if:
\begin{itemize}
\item
$\Gamma$ has no double edges or edges joining a vertex to itself (this is attained using residual finiteness of $\pi_1\Gamma$) and
\item
the path $(v_i)_{i=1}^n$ is always chosen to be geodesic.
\end{itemize}

It thus suffices to pass to $\Gamma$ where for each vertex $v_1$ there is a geodesic terminating at a boundary vertex $v_n$ that does not pass through a prescribed neighbor $v_0$ of $v_1$.

To do this, we take the following degree $2^k$ cover $\widehat{M}$ of $M$, where $k$ is the number of blocks of $M$. The cover $\widehat{M}$ is defined by the mapping of $H_1(M,\Z)$ into $\Z_2^k$ determined by the cohomology classes of closed non-separating vertical tori, one in each of the $k$ blocks.
Let $\widehat{\Gamma}$ the underlying graph of the graph manifold $\widehat{M}$. See Figure~\ref{fig:graph}.

\begin{figure}[ht]
\includegraphics[width=10cm]{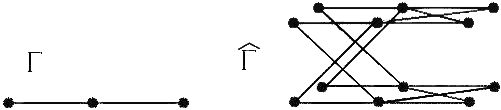}
\caption{the graphs $\Gamma$ and $\widehat{\Gamma}$}
\label{fig:graph}
\end{figure}

Fix vertices $v_0,v_1\in \widehat{\Gamma}$ and let $\gamma$ be a geodesic path in $\widehat{\Gamma}$ from $v_1$ to a boundary vertex $v_n$. If $\gamma$ passes through $v_0$, then we alter it as follows. Let $g$ denote the nontrivial element of the group of covering transformations of $\widehat{M}$ fixing $v_1$. Then $g$ maps $\gamma$ to a geodesic path disjoint from $v_0$ terminating at a boundary vertex. This shows that $\widehat{M}$ satisfies (antennas) and completes the proof of Proposition~\ref{prop:suff exists}.
\end{proof}

We record the following consequence of the proof of Proposition~\ref{prop:suff exists}, which will be used in \cite{PWmixed}.

\begin{cor}
Let $M$ be a graph manifold with nonempty boundary. There exists a finite cover  $\widehat{M}$ of $M$ such that
for each circle $C_0$ in a boundary torus $T\subset \widehat{M}$
there is an incompressible surface $S$ embedded in $\widehat{M}$ with $S\cap T$ consisting of a
nonempty set of circles parallel to $C_0$.
\end{cor}

\section{Separability: preliminaries}
\label{sec:block_separation}

This section prepares the background for the proofs of Theorems~\ref{thm:separability} and~\ref{thm:double_separability}.

\subsection*{Hempel's theorem}

We begin with discussing consequences of Hempel's theorem:
\begin{theorem}[special case of {\cite[Thm~1.1]{He}}]
\label{thm:Hempel}
Fundamental groups of graph manifolds are residually finite.
\end{theorem}

\begin{cor}
\label{cor:vertical separable}
If $T$ is an incompressible vertical torus in a block of a simple graph manifold $M$, then
$\pi_1T$ is separable in $\pi_1M$. If $T$ is a JSJ or boundary torus, then any finite index
subgroup of $\pi_1T$ is separable in $\pi_1M$.
\end{cor}

The proof of Corollary~\ref{cor:vertical separable} uses characteristic covers. The \emph{$n$--characteristic} cover of a manifold $B$
is the finite cover corresponding to the intersection of all subgroups of $\pi_1B$ of index $n$.
For example, the $n$--characteristic cover of a torus $T$ corresponds to $n\Z\times n\Z\subset \Z\times \Z=\pi_1T$.
If $B$ is a simple block, then since it retracts onto its boundary tori, its
$n$--characteristic cover restricts to $n$--characteristic covers over its boundary tori.

\begin{proof}[Proof of Corollary~\ref{cor:vertical separable}]
The group $\pi_1T$ is separable in $\pi_1M$ since it is a maximal abelian subgroup and
$\pi_1M$ is residually finite. A finite index subgroup $H\subset \pi_1T=\Z\times \Z$
is contained in some $n\Z\times n\Z$. Hence it suffices to consider a finite cover $M'$ of $M$
formed by gluing $n$--characteristic covers of the blocks. Since $n\Z\times n\Z$ is separable in
$\pi_1M'$, we have that $H$ is separable in $\pi_1M$.
\end{proof}

We now prove the analogous result for annuli.

\begin{cor}
\label{cor:vertical_annuli}
Let $T$ be a JSJ or boundary torus in a graph manifold $M$. Then every cyclic subgroup
of $\pi_1T$ is separable in $\pi_1M$.
\end{cor}

\begin{proof}
Let $\Z$ be a cyclic subgroup of $\pi_1T$ and let $g\in \pi_1M-\Z$. There is a finite index subgroup $H\subset\pi_1T$
containing $\Z$, but not $g$. We apply Corollary~\ref{cor:vertical separable} to $H$.
\end{proof}

Corollary~\ref{cor:vertical separable} has two further consequences.

\begin{cor}
\label{cor:straight}
Let $S\subset M$ be an incompressible surface in a graph manifold. Then there is a finite cover of $M$ where each elevation of $S$ is \emph{straight}, in the sense that its vertical annular pieces always join two distinct boundary components of the block.
\end{cor}

\begin{cor}
\label{cor:gamma does not backtrack}
Let $\gamma$ be a path in a graph manifold $M$ such that its lift $\widetilde{\gamma}$ to the universal cover $\widetilde{M}$ passes through as few blocks of $\widetilde{M}$ as possible in its path-homotopy class. Then there is a finite cover $M'$ of $M$, where $\gamma'$ (the quotient of $\widetilde{\gamma}$) does not pass through the same JSJ torus more than once.
\end{cor}


\subsection*{Untwining arcs}
The following result will play a role in the proof of Theorem~\ref{thm:double_separability}.

\begin{proposition}
\label{prop:Hempel_simult}
Let $F$ be a surface with two distinguished boundary components $C_1,C_2$ joined by an embedded arc $\alpha$. Let $\mathcal{A}$ be a finite family of arcs properly embedded in $F$ with endpoints on $C_1$ and $C_2$. Then for each sufficiently large $n$ there is a cover $F^*$ of $F$ of degree $n!$ on each boundary component and satisfying the following:
\begin{itemize}
\item[(*)] There is a lift $\alpha^*$ of $\alpha$ such that, if $C^*_i$ are the elevations of $C_i$ through the endpoints of $\alpha^*$, then all of the lifts of the arcs in $\mathcal{A}$ starting from $C^*_i$ do not terminate at $C^*_j$, except possibly for arcs homotopic to $\alpha^*$ relative to the boundary circles (if $\mathcal{A}$ contains arcs homotopic to $\alpha$).
\end{itemize}
\end{proposition}

In the proof, we will need the following ``omnipotency'' lemma.
\begin{lemma}
\label{lem:omnipotent}
Let $F$ be a surface of non-zero genus. Assign a number $n_i>0$ to each boundary component $C_i$ of $F$. Then there is a finite cover $F^*$ of $F$ having degree $n_i$ on each component of the preimage of $C_i$.
\end{lemma}
\begin{proof}
Since $F$ has non-zero genus, there is a non-separating simple closed curve $\beta\subset F$. Take the double cover $F_d$ determined by the cohomology class $[\beta]\in H^1(F,\Z_2)$. Each boundary component $C_i$ of $F$ lifts to a pair of boundary components $C^1_i,C^2_i$ of $F_d$. Choose a family of disjoint simple arcs $\beta_i$ joining $C^1_i$ to $C_i^2$. Take the cover $F^*$ determined by the mapping of $H_1(F_d,\Z)$ to $\prod \Z_{n_i}$ determined by the cohomology classes $[\beta_i]\in H^1(F_d,\Z_{n_i})$.
\end{proof}
\begin{rem}
\label{rem:cover good for alpha}
There is an extra feature to the construction in the proof of Lemma~\ref{lem:omnipotent}. If $\alpha\subset F$ is an arc joining two distinct boundary components of $F$, then no two lifts of $\alpha$ to $F^*$ join the same pair of boundary components. See Figure~\ref{fig:alphalifts}.
\end{rem}

\begin{figure}[t]
\begin{center}
\includegraphics[width=12cm]{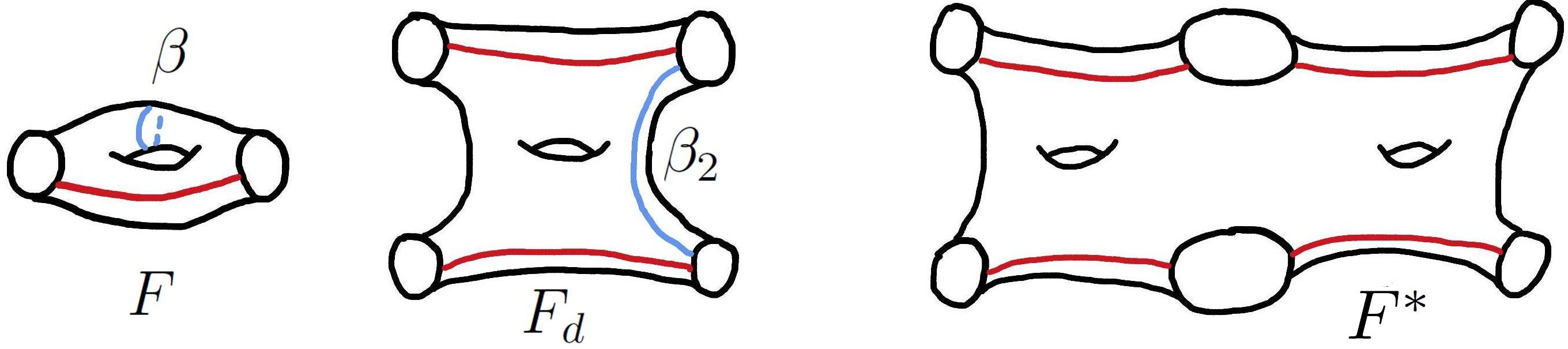}
\caption{$\beta,\beta_2$ and the lifts of $\alpha$ to $F_d$ and $F^*$ with $d_1=1, d_2=2$
\label{fig:alphalifts}}
\end{center}
\end{figure}

\begin{proof}[Proof of Proposition~\ref{prop:Hempel_simult}]
First we find $F^*$ satisfying property (*) but possibly having the wrong degree on the boundary components. We consider $H=\pi_1(C_1\cup\alpha\cup C_2)\subset\pi_1F$. Since $H$ is a finitely generated subgroup of the free group $\pi_1F$, it is separable. We can assume that the endpoints of all the arcs in $\mathcal{A}$ actually coincide with the endpoints of $\alpha$. Let $A\subset \pi_1F$ be the finite set of elements determined by $\alpha\cup\alpha_i^{-1}$, for those arcs $\alpha_i\in \mathcal{A}$ which are not homotopic to $\alpha$ relative to the boundary circles. The set $A$ is disjoint from $H$, so there is a finite index subgroup $G$ of $\pi_1F$ containing $H$ but disjoint from the finite set $A$. The cover $F^*$ corresponding to $G$ satisfies~(*).

Let $n$ be so large that $n!$ is divisible by all the degrees on the boundary components of the cover $F^*\rightarrow F$. Applying Lemma~\ref{lem:omnipotent} we pass from $F^*$ to a further cover having degrees $n!$ on the boundary. Note that if $\mathcal{A}$ does not contain an arc homotopic to $\alpha$, then (*) persists under passage to a finite cover, and we are done. Otherwise, if $\alpha$ lies in $\mathcal{A}$, we must additionally invoke Remark~\ref{rem:cover good for alpha}.
\end{proof}

\subsection*{Surface-injective covers}

In the proof of Theorems~\ref{thm:separability} and~\ref{thm:double_separability} we will need ``surface-injective'' covers.
As preparation, we discuss the structure of the following infinite cover. As usual, we assume that $S\subset M$ is an incompressible oriented surface embedded in a graph manifold $M$, and that each nonempty intersection of $S$ with a block of $M$ is either vertical or horizontal.

\begin{defin}
\label{def:MS}
Let $M^S$ denote the infinite cover $\pi_1S\backslash\widetilde{M}$ of $S$ corresponding to $\pi_1S\subset \pi_1M$. Let us describe the topology of non-empty
blocks of $M^S$. For each horizontal component $S_0$ of $S\cap M_v$, there is in $M^S$ an associated horizontal block $M_v^{S_0}\cong S_0\times \R$. Similarly, for
each vertical component $S_0$ of $S\cap M_v$, there is in $M^S$ an associated vertical block $M_v^{S_0}\cong S^1\times \widetilde{F}_v$. The annulus $S_0$ embeds inside $M_v^{S_0}$ as a product of the factor $S^1$ and a proper arc on $\widetilde{F}_v$.
\end{defin}

\begin{defin}
\label{def:separating cover}
Let $S$ be a surface in $M$. A finite cover $M'$ of $M$ is called \emph{$S$--injective}
with respect to the JSJ tori, if $S$ lifts to $M'$ and $S\cap B'$ is connected for each block $B'$ of $M'$. Moreover, we require that each horizontal component of
$S\cap B'$ maps with degree $1$ onto the base surface of $B'$. We allow $M'=M$.
\end{defin}

In particular, $M'$ arises from $M^{S}$ by quotienting each non-empty block
separately (though empty blocks are identified). Observe that the intersection of the lift
of $S$ with each JSJ or boundary torus of $M'$ is connected.
The property of being $S$--injective
is not preserved under passage to covers. Nevertheless, in Construction~\ref{constr:separating} we provide (high-degree) $S$--injective covers. We need the following
terminology and lemma:

\begin{defin}
\label{def:partial}
A \emph{semicover} of a graph manifold $M$ with respect to the JSJ tori is a graph manifold $\overline{M}$ together with a local embedding $\overline{M}\rightarrow M$ restricting to a covering map over each JSJ torus and over each open block. We say that the semicover is \emph{finite} if $\overline{M}$ is compact. Then $\overline{M}\rightarrow M$ can only fail to be a covering map at a torus $\overline{T}$ of $\partial\overline{M}$ that covers a JSJ torus of $M$. We refer to such a $\overline{T}$ as a \emph{halt torus}.
\end{defin}

\begin{lem}
\label{lem:completions}
Let $p\colon\overline{M}\rightarrow M$ be a finite semicover. Suppose that all halt tori of $\overline{M}$ map homeomorphically onto JSJ tori of $M$. Then we can embed $\overline{M}$ in a graph manifold $M'$ such that the semicover $p$ extends to a finite cover $M'\rightarrow M$.
\end{lem}
\begin{proof}
For each $M_v$ let $d_v$ be the degree of the (possibly disconnected) cover $p^{-1}(M_v)\rightarrow M_v$. Similarly, let $d_{v,w}$ be the degree of $p^{-1}(T)\rightarrow T$ for the torus $T=M_v\cap M_w$. Let $D=\max_{v,w} \{d_{v,w}\}$. For each $v$, take $D-d_v$ copies of $M_v$ and glue these copies to $\overline{M}$ to form $M'$.
\end{proof}

\begin{constr}
\label{constr:separating}
Let $S$ be a straight incompressible surface in $M$ (see Corollary~\ref{cor:straight})
and let $N>0$ be divisible by all the degrees of (possibly disconnected)
covering maps $S\cap M_v\rightarrow F_v$. Then there is a finite cover $M^S_N$ of $M$
which is
\begin{itemize}
\item
$S$--injective with respect to the JSJ tori and
\item
such that each JSJ or boundary torus of $M^S_N$ intersected by the lift of $S$ maps to a torus
$T$ of $M$ with degree $\frac{N}{|S\cap T|}$.
\end{itemize}
\end{constr}

Here $|S\cap T|$ denotes the number of connected components of $S\cap T$. The construction also
works if $S$ is disconnected and we will need this in \cite{PWsep}.

\begin{proof}
Consider a horizontal block $M_v$. Let $n=|S\cap M_v|$. Let $S_0$ be a component
of $S\cap M_v$ (they are all parallel). We take the unique degree $\frac{N}{n}$ cover of
$M_v$ to which $S_0$ lifts. It is the quotient of the $M^{S_0}_v$ block of $M^S$
(see Definition~\ref{def:MS}) by the $\frac{N}{n}$th power of the generator of covering
transformations. By the divisibility hypothesis, the result of this over the boundary of $M_v$ is that: if there are $k$ components of intersection of $S$ with a boundary torus, then in the cover we get $k$ tori components projecting with degree $\frac{N}{k}$. Hence for two adjacent horizontal blocks we have a matching between the elevations of the JSJ tori crossed by (the lifts of) $S_0$'s.

Now consider a vertical block $M_v$. Fix a component $S_0$ of $S\cap M_v$. Let $F'_v$ be the
double cover of $F_v$ determined by the $\Z_2$ cohomology class of a non-separating simple closed curve on $F_v$. Each boundary component $C$ of $F_v$ is covered in $F'_v$ by a pair $C_1,C_2$. Suppose $S_0$ connects tori $T,Q$ with base boundary circles of $F_v$ denoted by $C^T,C^Q$. Let $t=|S\cap T|$ and $q=|S\cap Q|$. Pick disjoint embedded arcs $\theta, \omega$ in $F_v'$ joining $C^T_1$ with $C^T_2$ and $C^Q_1$ with $C^Q_2$. Take $\frac{N}{t}-1$ extra copies of $F_v'$ containing copies of $\theta$ and $\frac{N}{q}-1$ extra copies of $F_v'$ containing copies of $\omega$. Cutting and regluing along these arcs in cyclic order gives a cover of $F_v$ whose boundary components project homeomorphically, except two degree $\frac{N}{t}$ covers of $C^T$ and two degree $\frac{N}{q}$ covers of $C^Q$. To get a cover of $M_v$ we form the product with $S^1$.

We now take one such covering block for each component $S_0$ of $S\cap M_v$ for vertical $M_v$ and two blocks described above for horizontal $M_v$. All boundary components match except that there are some hanging boundary components giving rise to halt tori.

This concludes the construction of a semicover. Note that the lift of $S$ crosses all blocks
of this semicover that cover vertical blocks of $M$ and half of the blocks covering a horizontal
one. This semicover satisfies the hypothesis of Lemma~\ref{lem:completions}. We use it to
obtain a (non-unique) $S$--injective cover $M_N^S$.
\end{proof}

\section{Separability of a surface}
\label{sec:separability}
\begin{proof}[Proof of Theorem~\ref{thm:separability}]
We can assume that $M$ is simple (see Section~\ref{sec:not}) and $S$ is straight (see Corollary~\ref{cor:straight}).

If $S$ is a vertical torus or annulus contained in a single block, then $\pi_1S$ is
separable by
Corollaries~\ref{cor:vertical separable} and~\ref{cor:vertical_annuli}. Otherwise, $S$ contains a horizontal piece. Choose the basepoint $\tilde{m}$ of the universal cover $\widetilde{M}$ of $M$ in the interior of a horizontal piece of the elevation $\widetilde{S}$ of $S$ to $\widetilde{M}$ stabilized by $\pi_1S\subset \pi_1M$. Let $g\in \pi_1M-\pi_1S$ and let $\widetilde{\gamma}$ be a path in $\widetilde{M}$ representing $g$, i.e.\ joining $\tilde{m}$ to $g\tilde{m}$. Then $\widetilde{\gamma}$ does not end on $\widetilde{S}$. Our goal is to find a finite cover with the same property.

We can assume that $\widetilde{\gamma}$ crosses as few elevations of JSJ tori as possible (in other words, $\widetilde{\gamma}$ does not ``backtrack'' in $\widetilde{M}$.) Let $\widetilde{B}$ denote the last non-empty block of $\widetilde{M}$ entered by $\widetilde{\gamma}$ (when all blocks crossed by $\widetilde{\gamma}$ are non-empty, we take $\widetilde{B}$ to be the last one).

We first consider the case where $\widetilde{B}$ is the last block of $\widetilde{M}$ entered by $\widetilde{\gamma}$.
Then $\widetilde{B}$ is horizontal (by the choice of $\tilde{m}$).
In the quotient $B^S\subset M^S$ of $\widetilde{B}$ the projection of the endpoint of $\widetilde{\gamma}$ is still disjoint from $S$ and the same is true in a sufficiently large cyclic quotient of the block $B^S$. This quotient coincides with an appropriate block of the cover $M_N^S$ from Construction~\ref{constr:separating}. Hence for $N$ sufficiently large the cover $M_N^S$ is as desired.

We now consider the case where $\widetilde{B}$ is not the last block entered by $\widetilde{\gamma}$, in which case $\widetilde{B}$ is vertical. By Corollary~\ref{cor:gamma does not backtrack} we can pass to a finite cover $M'$ where the projection $\gamma'$ of $\widetilde{\gamma}$ does not backtrack, i.e.\ $\gamma'$ does not cross the same JSJ torus twice. This property will be preserved under taking further covers.

Let $S'$ denote the elevation of $S$ to $M'$.
For separability of $\pi_1S$ we shall prove that $\gamma'$ does not end in $S'$ within $M'$ or after passing to a further finite cover.

Let $\widetilde{T}$ denote the universal cover of a JSJ torus through which $\widetilde{\gamma}$ leaves $\widetilde{B}$. Let $T'$ be its quotient in $M'$.
Our first step is to guarantee that in $M'$ (or its finite cover), the surface $S'$ does not cross $T'$.

\begin{figure}[ht]
\begin{center}
\includegraphics[width=9cm]{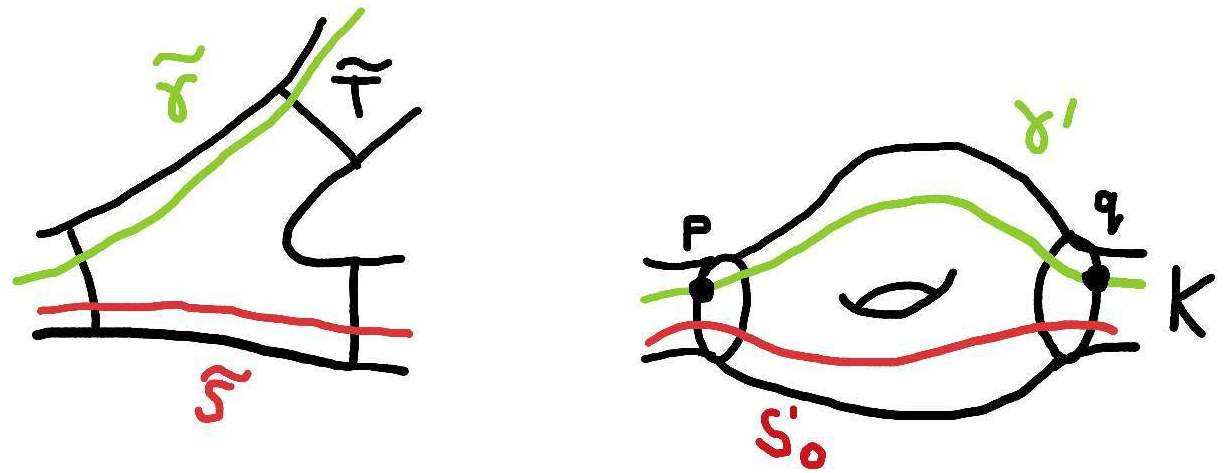}
\end{center}
\caption{$q$ requires removing from $K$}
\label{fig:pq}
\end{figure}

The quotient block $B'\subset M'$ of $\widetilde{B}$ is vertical.
Let $p$ and $q$ denote the projections to $B'$ of the first and last point of the intersection
of $\widetilde{\gamma}$ with $\widetilde{B}$.
Let $S'_0$ be the quotient in $B'$ of $\widetilde{S}\cap \widetilde{B}$ (there
might be some other components of $S'\cap B'$). Let $K$ be the JSJ torus crossed by $S'_0$
other than the one containing $p$. Any further $S'$--injective cover (for example $M'^{S'}_N$
from Construction~\ref{constr:separating}) satisfies our condition $S'\cap T'=\emptyset$
unless $q\in K$ (i.e.\ $T'=K$), see Figure~\ref{fig:pq}. In that case we first use
Corollary~\ref{cor:vertical separable} to pass to a cover where (keeping the same notation) the point $q$ does not lie in $K$, and so $S'_0$ does not cross $T'$. There might still be an accidental component of $S'\cap B'$ intersecting $T'$. We can remove it by passing to an $S'$--injective cover.

Summarizing, we have constructed a cover $M'$ where $T'$ is disjoint from $S'$. It suffices now to pass to a degree $2$ cover determined by the cohomology class $[T']\in H^1(M',\Z_2)$. In that cover the portion of $\gamma'$ after $q$ is contained entirely in the union of empty blocks. In particular, its end lies outside the appropriate lift of $S'$, as desired.
\end{proof}

\section{Separability of intersecting surfaces}
\label{sec:double}

\textbf{Outline of the argument.}
Let $\widetilde{S}$ and $\widetilde{P}^0$ be intersecting elevations of $S$ and $P$ to the universal cover $\widetilde{M}$ of $M$. We reserve the notation $\widetilde{P}$ for a different elevation of $P$. By our hypothesis, $\widetilde{S}\cap\widetilde{P}^0$ is non-empty, and so we can choose the basepoint $\tilde{m}\in \widetilde{S}\cap\widetilde{P}^0$.

We fix $g\in \pi_1M-\mathrm{Stab} (\widetilde S) \mathrm{Stab} (\widetilde P^0)$ and take a path $\widetilde{\gamma}$ starting at $\tilde{m}$ representing $g$ in $\widetilde{M}$. Let $\widetilde{P}$ denote the elevation of $P$ through the terminal point $g\tilde{m}$ of $\widetilde{\gamma}$. We aim to find a finite quotient of $\widetilde{M}$, where the projections of $\widetilde{S}$ and $\widetilde{P}$ are ``disjoint'' in the sense that they do not intersect at a basepoint-translate.

The main object we work with is the ``core'' $\overline{M}\subset\widetilde{M}$ consisting of blocks simultaneously intersecting $\widetilde{S}$ and $\widetilde{P}$. In Step~1 we prove that for each $\pi_1S$ orbit in $\overline{M}$ of a
core block, we can quotient it to a finite block with ``disjoint'' quotients of $\widetilde{S}$ and $\widetilde{P}$.

In Step~2 we use Step~1 to show how to simultaneously quotient the whole core (or rather its extension) to a finite quotient $\widehat{M}'$ where the images of $\widetilde{S}$ and $\widetilde{P}$ are ``disjoint'' (Step~2(i)). Moreover, we arrange that the images $\widetilde{S}$ and $\widetilde{P}$ never simultaneously intersect the same halt torus of the semicover (see Definition~\ref{def:partial}) $\widehat{M}'\rightarrow M$ (Step~2(iii)). The semicover $\widehat{M}'$ extends to a finite cover $M'$ by Step~2(ii).

Finally, in Step~3, we use Step~2(iii) to pass to a further cover, where the quotients of $\widetilde{S}$ and $\widetilde{P}$ can meet only inside the image of the core. But this is excluded by Step 2(i).

\smallskip

\begin{figure}[ht]
\begin{center}
\includegraphics[width=5cm]{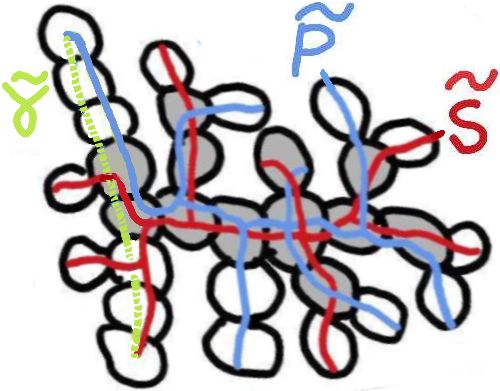}
\end{center}
\caption{core of $\widetilde{M}$
\label{fig:core}}
\end{figure}

\begin{proof}[Proof of Theorem~\ref{thm:double_separability}]
We choose $\tilde{m}\in \widetilde{S}\cap\widetilde{P}^0$ as in the outline of the argument.
Without loss of generality, we can assume that if $\tilde{m}$ lies in a
vertical piece of $\widetilde{P}^0$ then it also lies in a vertical piece of $\widetilde{S}$. We identify $\mathrm{Stab}(\widetilde{S})\subset\pi_1M$ with $\pi_1S$. Note that by Corollary~\ref{cor:straight}, by passing to a finite cover we can assume that $S$ and $P$ are straight. As usual, $M$ can be also assumed to be simple and $S$--injective.

Let $g\in \mathrm{Stab} (\widetilde S) \mathrm{Stab} (\widetilde P^0)$ and let $\widetilde{\gamma}$ be a path representing $g$ in $\widetilde{M}$ as in the outline. We can assume that $\widetilde{\gamma}$ traverses as few blocks of $\widetilde{M}$ as possible.

Recall that $\widetilde{S}$ is the elevation of $S$ to $\widetilde{M}$ passing through the initial point $\tilde{m}$ of $\widetilde{\gamma}$, and $\widetilde{P}$ is the elevation of $P$ to $\widetilde{M}$ passing through the terminal point $g\tilde{m}$ of $\widetilde{\gamma}$. Our hypothesis on $g$ says that $\widetilde{S}$ and $\widetilde{P}$ do not cross at any translate
of the basepoint $\widetilde{m}$ (we will refer to such a point or its quotient in an intermediate cover as a \emph{basepoint-translate}). Our separability goal is to find a finite cover of $M$ with the same property.

The \emph{core} $\overline{M}$ of $\widetilde{M}$ is the union of blocks intersecting both $\widetilde{S}$ and $\widetilde{P}$ (see Figure~\ref{fig:core}, but note that the core might consist of an infinite number of blocks). Assume that the core is non-empty, we will consider the other case at the very end of the proof.

Let $\overline{S}=\widetilde{S}\cap \overline{M}$ and $\overline{P}=\widetilde{P}\cap \overline{M}$. Let $\overline{M}^S$ be the manifold obtained from the core by identifying points in the same orbit of $\pi_1S$ (this is not a genuine action on $\overline{M}$, only a partial one). In this identification we
treat the core as an open manifold, so we do not identify boundary components whose adjacent core blocks are not identified.

\begin{figure}[ht]
\includegraphics[width=8cm]{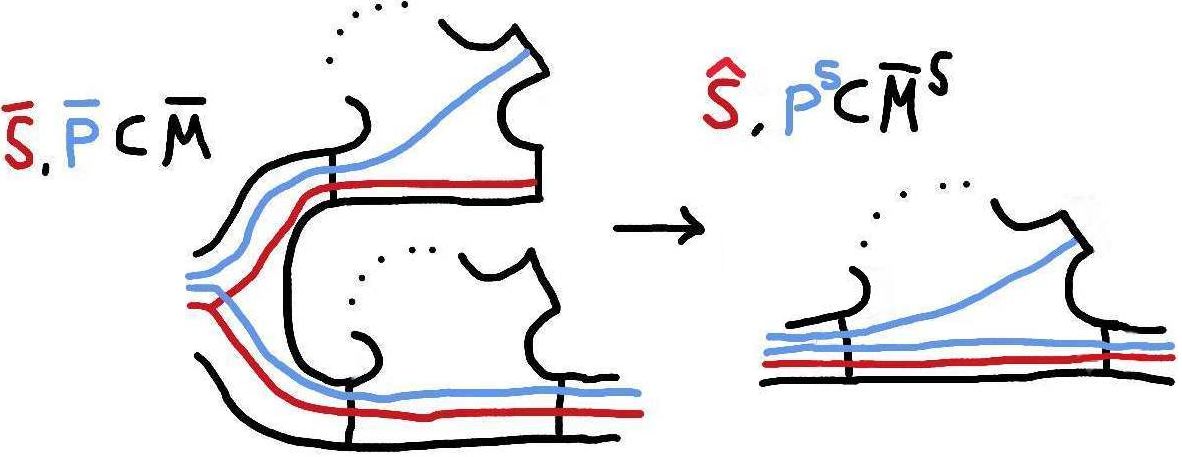}
\caption{$\overline{P}\rightarrow P^S$ is proper}
\label{fig:proper}
\end{figure}

Let $P^S$ be the quotient of $\overline{P}$ in $\overline{M}^S$ and let $\widehat{S}$ be the quotient of $\overline{S}$ in $\overline{M}^S$. Note that $P^S$ and $\widehat{S}$ do not go through the same basepoint-translate. The notation $\widehat{S}$ instead of $S^S$ is justified by the fact that $\widehat{S}$ is in fact a lift of a ``core'' subsurface of $S$.
Note that though the map $\overline{M}\rightarrow \overline{M}^S$ is not \emph{proper} in the sense that a boundary component of $\overline{M}$ might be mapped into the interior $\overline{M}^S$, its restriction to $\overline{P}\rightarrow P^S$ is proper (see Figure~\ref{fig:proper}). Equivalently, boundary components of $\overline{P}$ are mapped onto boundary components of $P^S$.

\begin{step} Let $B^S$ be a block in $\overline{M}^S$ covering a block $B$ of $M$. Then $B^S$ factors through a finite cover $B^*$ of $B$ where quotients of $P^S$ and $\widehat{S}$ still do not intersect at a basepoint-translate.
\end{step}

Loosely speaking, in Step~1 we shall achieve separability at a single core block.
Note that the property of the finite cover $B^*$ in Step~1 is preserved by passing to a further cover which is a quotient of $B^S$.

\smallskip
First assume that $\tilde{m}$ lies in a horizontal piece of $\widetilde{P}^0$. Consequently, if a basepoint-translate lies in $B^S$, then the block $B^S$ is $P^S$--horizontal. Let $S_0=\widehat{S}\cap B^S$. First consider the case where $S_0$ is vertical. Then there are only finitely many elevations of $P\cap B\subset M$ to $B^S$: their number is bounded by the degree of $P\cap B\rightarrow F$, where $F$ is the base surface of $B$. The action of covering transformations of $B$ on the universal cover of the block $B^S$ factors trough an action on $B^S$. As there are only finitely many elevations of $P\cap B$ to $B^S$, a finite index subgroup of the group of covering transformations preserves all of them. We quotient by this subgroup to obtain a desired finite cover $B^*$ of $B$.

Now consider the case where $S_0$ is horizontal. Then the action of covering transformations on the universal cover of the block $B$ factors to $\Z=\langle c\rangle$ action on $B^S$. The easy subcase is where one (hence any) elevation of $P\cap B$ to $B^S$ is non-compact (see Figure~\ref{fig:spiral}). Then as in the previous case, there are only finitely many elevations of $P\cap B$ and we can choose a finite cover $B^*$ obtained by quotienting by a subgroup $\langle c^k\rangle$ that maps $P^S\cap B^S$ onto itself.

\begin{figure}[]
\begin{center}
\includegraphics[width=2.5cm]{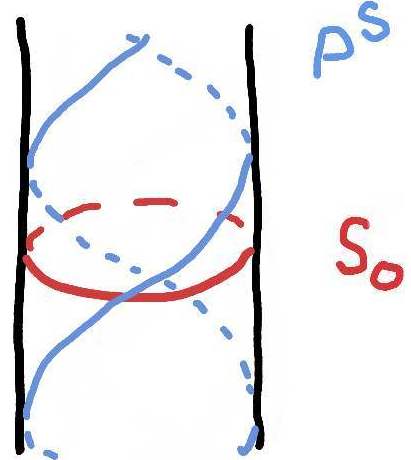}
\end{center}
\caption{noncompact components of $P^S\cap B^S$: cross-section by a cylinder}
\label{fig:spiral}
\end{figure}

The interesting subcase is where the elevations of $P\cap B$ to $B^S$ are compact. Let $P_0$ be any component of $P^S\cap B^S$. Let $\mathcal{P}$ be the maximal connected subsurface of $P^S$ containing $P_0$ consisting uniquely of horizontal pieces. First consider the situation where $P_0$ is properly contained in $\mathcal{P}\cap B^S$. We will show that $P^S\cap B^S$ is invariant under some $c^k$ (as in the case of non-compact $P_0$).

Note that the pieces of $\mathcal{P}\cap B^S$ might not lie in one $\langle c \rangle$--orbit. However, we can extend this action to another cyclic action $\langle c \rangle\subset \langle \underline{c} \rangle$ on $B^S$ by homeomorphisms for which all the pieces of $\mathcal{P}\cap B^S$ lie in one orbit.

Assume that for some $k\neq 0$ the translate $\underline{c}^kP_0$ lies in $\mathcal{P}$. Let $B_i$ be a sequence of blocks of $\overline{M}^S$ with
$P_i\subset B_i$ a sequence of pieces in $\mathcal{P}$ connecting $P_0$ to $\underline{c}^kP_0$. The action of $\underline{c}$ extends to all blocks $B_i$ crossed by $\mathcal{P}$ (some might be $\widehat{S}$--vertical). Hence for any $n$ there is a sequence of pieces that are translates of $P_i$ joining $\underline{c}^nP_0$ to $\underline{c}^{k+n}P_0$.
Thus $\underline{c}^nP_0$ lies in $P^S$ if and only if $\underline{c}^{k+n}P_0$ does, in other words $P^S\cap B$ is $\underline{c}^k$--invariant. Then for some $k'$ we have that $P^S\cap B$ is $c^{k'}$--invariant, as desired.

It remains to consider the situation where $P_0$ equals $\mathcal{P}\cap B^S$. Then, by the $c$--action argument above, the same is true for any choice of $P_0$ in $P^S\cap B^S$. Moreover, since there are only finitely many vertical pieces of $P^S$ in $\overline{M}^S$ with both boundary components in the interior of $\overline{M}^S$, only finitely many translate copies of $\mathcal{P}$ are joined together contributing to $P^S\cap B^S$. We conclude that $P^S\cap B^S$ is compact. Then for any sufficiently large $k$, no element of $\langle c^k \rangle$ maps a basepoint-translate in $S_0$ onto a point in $P^S$. This finishes the argument for Step~1 under the assumption that $\tilde{m}$ lies in a horizontal piece of $\widetilde{P}^0$.

\smallskip

Finally, assume that $\tilde{m}$ lies in the intersection of a vertical piece of $\widetilde{P}^0$ and a vertical piece of $\widetilde{S}$. Consequently, if a basepoint-translate lies in $B^S$, then $B^S$ is both $P^S$--vertical and $\widehat{S}$--vertical. Let $K_1,K_2$ be the boundary cylinders of $B^S$ crossed by $\widehat{S}$. By the definition of the core, except for the exceptional situation where the core is a single block and hence $P^S\cap B^S$ has just one component, each piece of $P^S\cap B^S$ intersects some $K_i$. By the $c$--action argument applied to adjacent ($P^S$--horizontal) blocks of $B^S$, for each $i=1,2$, the intersection $P^S\cap K_i$ is either compact or periodic. Then after quotienting $B^S$ by finite index subgroups of one, both or none of the stabilisers of $K_i$ we obtain $\check{B}$, in which the quotient of $P^S$ is compact and still does not intersect $\widehat{S}$ in a basepoint-translate.
Let $\check{F}\rightarrow F$ be the cover induced between the base surfaces of $\check{B}\rightarrow B$.
By separability of $\pi_1\check{F}$ in $\pi_1F$, the cover $\check{B}$ quotients further to a desired cover $B^*$. This completes the argument for Step~1.

\smallskip

Let $\widehat{M}$ denote the quotient of $\overline{M}$ (and hence $\overline{M}^S$) in $M$. However, if there is a JSJ torus $K$ in $M$ outside the image of the interior of $\overline{M}$ but with both of its adjacent blocks within the image of $\overline{M}$, then we put in $\widehat{M}$ two copies of $K$, each compactifying one of the adjacent blocks. In other words, $\widehat{M}$ is contained in $M$ only in the sense of manifolds open at the boundary. Since $M$ is $S$--injective, each block of $\widehat{M}$ is covered by exactly one block of $\overline{M}^S$.
We note, however, that $\overline{M}^S\rightarrow \widehat{M}$ is only an infinite semicover.

\begin{step}
There is a finite cover $\widehat{M}'$ of $\widehat{M}$ through which the map $\overline{M}\rightarrow \widehat{M}$ factors, with the following properties.
Let $\widehat{S}',\widehat{P}'\subset \widehat{M}'$ be the extensions of the quotients of $\overline{S}, \overline{P}$ in $\widehat{M}'$ to elevations of $S\cap \widehat{M}$ and a component of $P\cap \widehat{M}$.

\begin{enumerate}[(i)]
\item
Each block $B'$ of $\widehat{M}'$ is a further cover of a cover $B^*$ satisfying Step~1. Moreover, the images of $\widehat{S}'\cap B',\ \widehat{P}'\cap B'$ in $B^*$ are contained in the quotients of $P^S$ and $\widehat{S}$.
\item
Over all boundary tori the cover $\widehat{M}'\rightarrow \widehat{M}$ is
$n!$--characteristic (for some uniform $n$), i.e.\ it corresponds to
the subgroup $n!\Z\times n!\Z\subset \Z\times \Z$.
\item
Each halt torus of the semicover $\widehat{M}'\rightarrow M$ intersects at most one of $\widehat{S}',\widehat{P}'$.
\end{enumerate}
Moreover, $\widehat{M}'$ is $\widehat{S}'$--injective.
\end{step}

Note that in view of Step~1, Step~2(i) implies immediately that $\widehat{S}'$ and $\widehat{P}'$ do not intersect at a basepoint-translate.

\smallskip

The value $n$ is the maximum of $n$ needed to execute the following construction over each of the finitely many blocks of $\widehat{M}$.

\smallskip

First suppose that $B\subset \widehat{M}$ is $P$--horizontal and $S$--horizontal. Take the cyclic cover $B^*$ guaranteed by Step~1. It may be taken with any degree $n!$ for $n$ sufficiently large. To make the quotient to $B$
characteristic on the boundary tori, we pass from $B^*$ to a cover $B'$ induced by any cover of $S\cap B$ of degree $n!$ on each boundary component (use Lemma~\ref{lem:omnipotent}).

\smallskip

Now assume that $B\subset \widehat{M}$ is $P$--horizontal but $S$--vertical.
Again take the cover $B^*$ guaranteed by Step~1. By Lemma~\ref{lem:omnipotent} it may be chosen to be degree $n!$ on boundary tori for $n$ sufficiently large. In order to make the quotient to $B$
characteristic on the boundary tori, we pass from $B^*$ to a cyclic cover $B'$ of degree $n!$ determined by an arbitrary degree~$1$ horizontal surface.

\smallskip
In the case where the block $B\subset \widehat{M}$ is both $P$--vertical and $S$--vertical, finding convenient $B^*$ will involve several steps. First of all, there is a cover $B^*$ of $B$ satisfying Step~1. Since we still want to replace $B^*$ by a particular finite cover, in order to simplify the notation we will assume that $B$ already has the property from Step~1 that $S\cap P$ and $P\cap B$ do not intersect at a basepoint-translate. This property is invariant under taking covers, so any further cover $B^*$ that we will construct will satisfy Step~1.

First assume that we are in a (simpler) subcase where in $B^S$ (the block in $\overline{M}^S$ covering $B$) there is a vertical annular piece of $P^S$ homotopic to $\widehat{S}\cap B^S$. In that case we apply Proposition~\ref{prop:Hempel_simult}
with $F$ being the base surface of the block $B$. Let $\alpha\subset F$ be the base arc of $S\cap B$. Let $\mathcal{A}$ be the family of those base arcs of $P^S\cap B^S$ that intersect the same components of $\partial F$ as $\alpha$.

By Proposition~\ref{prop:Hempel_simult}, for sufficiently large $n$ we get a finite quotient $B^*$ of $B^S$ of degree $n!$ on all boundary components over $B$. The base surface of the cover $B^*$ satisfies also (*), which will be used later. As before, to get a
characteristic cover over the boundary tori, we pass to a cover $B'$ of $B^*$, determined by an arbitrary element of $H^1(B^*,\Z_{n!})$ dual to a degree one horizontal surface in $B^*$.

\smallskip

In the (harder) subcase where in $B^S$ there is no vertical annular piece of $P^S$ homotopic to $\widehat{S}\cap B^S$, in Proposition~\ref{prop:Hempel_simult} we want to use, instead of $B$, an intermediate finite quotient $B_{int}$ of $B^S$ also having the property that the projection of the vertical annular piece of $\widehat{S}$ is not homotopic to the projection of any piece of $P^S$. ($B$ might not have this property.)

Let $K_1,K_2$ be the boundary cylinders of $B^S$ crossed by $\widehat{S}$. By the argument of Step~1 applied to adjacent blocks of $B^S$, for each $i=1,2$ either $P^S\cap K_i$ is compact, or there is a subgroup $k_i\Z \subset \Z$ of the covering transformations of $K_i$ that preserves $P^S\cap K_i$. This proves that any finite quotient $B_{int}$ of $B^S$ such that $B_{int}\rightarrow B$ over the quotients of $K_i$ in $B$ is of degree sufficiently high and divisible by $k_1k_2$ is as required. Now we repeat the construction of $B^*$ using Proposition~\ref{prop:Hempel_simult} with $B_{int}$ in place of $B$, so the surface $F$ is the base surface of $B_{int}$. However, we require that the boundary degrees are $n!$ when we quotient $B^*$ to $B$, not to $B_{int}$. Again, we record property~(*) for later use. We obtain $B'$ from $B^*$ as before. This closes the discussion of the case, where $B$ is both $P$--vertical and $S$--vertical.

The diagram below illustrates the intermediate covers between the universal $\overline{B}\rightarrow B$.

\begin{diagram}
\overline{B} &\rTo&B'    &      &B_{int}\\
\dTo         &    &\dTo   &\ruTo &       &\rdTo\\
B^S          &\rTo&B^*&      &\rTo   & & B
\end{diagram}

The last case is where a block $B$ is $P$--vertical and $S$--horizontal. Here $n$ is arbitrary. We first take the degree $n!$ cyclic cover $B^*$ of $B$ determined by  $[S_0]\in H^1(B^*,\Z_{n!})$, where $S_0=S\cap B$. Next we pass to a cover $B'$ induced by any cover $S'_0\rightarrow S_0$ of degree $n!$ on each boundary component (use Lemma~\ref{lem:omnipotent}).

\smallskip

We take $n$ sufficiently large for both the construction in Step~1 and various applications of Proposition~\ref{prop:Hempel_simult} with the above data to work. Then all blocks $B$ have covers $B'$ that are
$n!$--characteristic on the boundary. Now we take the right number of copies of each $B'$ so that the degree of the disconnected cover from the union of the copies of $B'$ to $B$ does not depend on $B$.
In each of these copies we distinguish one elevation $\Sigma_c$ of (connected) $S\cap B$. In the case where $B$ is $S$--horizontal, the surface $\Sigma_c$ is of degree $1$ over the base surface of $B'$. Hence the intersection of $\Sigma_c$ with each boundary torus is at most a single curve. We match up these blocks, also matching the $\Sigma_c$'s,
to form a cover $\widehat{M}'$ of $\widehat{M}$. We pick any of the maps $\overline{M}\rightarrow \widehat{M}'$ mapping $\overline{S}$ to the union $\Sigma$ of the $\Sigma_c$'s. We will now verify that all the required properties of $\widehat{M}'$ hold.

\smallskip

Property (ii) is clear from the construction. Note that $\widehat{M}'$ is $\widehat{S}'$--injective, since $\widehat{S}'$ is contained in $\Sigma$.
Moreover, obviously for each copy $B_c$ in $\widehat{M}'$ of a block $B'$ that covers the quotient $B^*$ of $B^S$ in $\overline{M}^S$ we have the following. Under the identification of $B_c$ with $B'$, the projection to $B^*$ of the intersection $\widehat{S}'\cap B_c$ is contained in the projection to $B^*$ of $\widehat{S}\cap B^S$. We now claim the same for $\widehat{P}'$: the projection to $B^*$ of the intersection $\widehat{P}'\cap B_c$ is contained in the projection to $B^*$ of $P^S\cap B^S$.

Before we justify the claim, we note that although the map $\overline{M}\rightarrow \widehat{M}$ factors through $\widehat{M}'$, the image of $\overline{P}$ in $\widehat{M}'$ might be smaller than $\widehat{P}'$. It is a priori unclear where the extension $\widehat{P}'$ is located. We look for a surface $\Pi$ for $\widehat{P}'$ which replaces $\Sigma$ for $\widehat{S}'$ in the argument above.

Property (i) follows from the claim since $B^*$ was chosen as in Step~1.

\smallskip

To justify the claim, let $H$ be the set of elements $h\in\pi_1S$ preserving that copy of the universal cover of $\widehat{M}'$ in $\widetilde{M}$ which contains $\overline{M}$. In other words, $H=\pi_1\widehat{S}$. (Note that the cosets $\pi_1\widehat{S}'\backslash H$ correspond to different ways in which we could have defined the projection $\overline{M}\rightarrow\widehat{M}'$ above.)
Let $\Pi\subset \widehat{M}'$ be the union of the projections of $h\overline{P}$, over $h\in H$.
Obviously $\Pi$ has the property that for each $B_c\simeq B'$ as above the projection to $B^*$ of the intersection $\Pi\cap B_c$ is contained in the projection of $P^S\cap B^S$. Since the projection of $\overline{P}$ is contained in $\Pi$, in order to justify the claim, it remains to prove that $\Pi$ is a surface properly embedded in $\widehat{M}'$:

The quotient in $\widehat{M}'$ of a translate $h\overline{P}$ with $h\in H$ might fail to be proper only at a quotient of a boundary line $h\pi$ of $h\overline{P}$. If $\pi$ lies in $\overline{M}$ in the boundary of a $\overline{P}$--horizontal and $\overline{S}$--vertical block, then the quotient of $h\pi$ also lies in the boundary of $\widehat{M}'$ (since it is $\widehat{S}'$--injective). The only other possibility is that $\pi$ lies in the boundary of a $\overline{P}$--vertical and $\overline{S}$--vertical block $\overline{B}$:
Denote by $\psi$ the opposite boundary line in $\overline{B}$ of the piece of $\overline{P}$ containing $\pi$. Also, denote by $\sigma$ the boundary line of $\overline{S}\cap \overline{B}$ in the plane not containing $\psi$. The only boundary component of the quotient of $h\overline{B}$ in $\widehat{M}'$ which is possibly in the interior of $\widehat{M}'$, besides the one containing the quotient of $h\psi$, is the one containing the quotient of $h\sigma$. Assume then that $h\pi$ and $h\sigma$ hence also $\pi$ and $\sigma$ are mapped into the same JSJ torus of $\widehat{M}'$. Property (*) of Proposition~\ref{prop:Hempel_simult} then says that we were in the ``simpler subcase'' of the discussion above: in $B^S$ there is a vertical piece $p$ of $P^S$ homotopic to $\widehat{S}\cap B^S$. Moreover, the vertical annulus $p$ contains $\pi$ upon passing to the quotient $B^*$ guaranteed in Proposition~\ref{prop:Hempel_simult}. See Figure~\ref{fig:piecep}.
Let $f\overline{B}$ be the block of $\overline{M}$ from which $p$ arises, with $f\in H$. In $f\overline{B}$ the pieces of $\overline{S}, \overline{P}$ are also homotopic. By the definition of the core, the line $f\pi$ is in the interior of $\overline{P}$. Hence $h\pi$ is in the interior of $hf^{-1}\overline{P}\subset \Pi$. This finishes the argument for the claim and hence for property~(i).

\begin{figure}[t]
\begin{center}
\includegraphics[width=12cm]{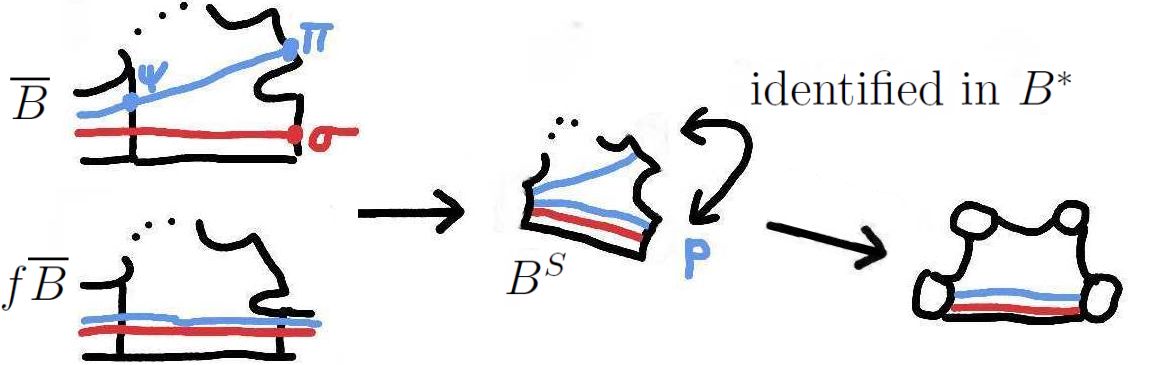}
\end{center}
\caption{finding $p$}
\label{fig:piecep}
\end{figure}

\smallskip
As for property (iii), we also need to use conclusion~(*) of Proposition~\ref{prop:Hempel_simult}. Let $K'$ be a halt torus of $\widehat{M}'$ in a copy of a block $B'$. Let $B^S\subset\overline{M}^S$ be the block mapped to the same $B^*$ as $B'$ and let $K^S$ be that elevation from $B^*$ to $B^S$ of the quotient of $K'$, which crosses $\widehat{S}$. Then $K^S$ lies also in the boundary of $\overline{M}^S$. Hence $P^S$ is disjoint from $K^S$ by the definition of the core. In view of (*) in the ``harder subcase'', the quotient of $\overline{P}$ in $\widehat{P}'$ is disjoint from $K'$. The same is true for $h\overline{P}$ over $h\in H$ ($H$ as in the proof of the claim above), hence for the whole $\widehat{P}'$. Thus we have proved property~(iii), that $\widehat{S}'$ and $\widehat{P}'$ do not cross $K'$ simultaneously. This completes the argument for Step~2.

\smallskip

The graph manifold $\widehat{M}'$ is a semicover of $M$. By Step~2(ii) we can complete it to a cover $M'$ by taking an appropriate number of disjoint copies of any finite covers of blocks in $M$ that are $n!$--characteristic
on the boundary. We require that $\widehat{M}'$ embeds in $M'$ as a closed submanifold --- we do not allow accidental matching of boundary components of open $\widehat{M}'$. By choosing those covers correctly we keep $M'$ to be $S'$--injective, where $S'\subset M'$ is the appropriate elevation of $S$. It remains to perform the following:

\begin{step}
There is a finite cover $M''$ of $M'$,
whose blocks $B''$ intersecting simultaneously the quotients $S'',P''$ of $\widetilde{S}, \widetilde{P}$ project to $B'\subset \widehat{M}'$ so that
$S''\cap B''$ maps into $\widehat{S}'$ and $P''\cap B''$ maps into $\widehat{P}'$.
\end{step}

Let $P'\subset M'$ be the quotient of $\widetilde{P}$. Let $\check{M}$ be a $P'$--injective cover of $M'$. We keep the notation
$P'$ for the lift of $P'$ to $\check{M}$ (quotient of $\widetilde{P}$) and
denote by $\check{S}$ the appropriate elevation of $S$ (quotient of
$\widetilde{S}$).
Let $\tau$ be the union of the JSJ tori of $\check{M}$ containing the
boundary components of $\widehat{P}'$ which are not in the boundary of $P'$.

We consider the degree 2 cover $M''$ of $\check{M}$ defined by the $\Z_2$
cohomology class $[\tau]$.
The union of tori $\tau$ is disjoint from $\check{S}$, by Step~2(iii). On
the other hand, by $P'$--injectivity, $\tau$ separates $P'$ into $\widehat{P}'$ and its complement.
Hence both $\check{S}$ and $P'$ lift to $M''$ and any piece of the lifted
$P'\setminus\widehat{P}'$ is in an $\check{S}$--empty block of $M''$.
This implies that $M''$ satisfies Step~3.

\smallskip

\textbf{Conclusion.} By Step~3, if surfaces $S''$ and $P''$ intersect in a block $B''$ of $M''$, then they project to surfaces $\widehat{S}'$ and $\widehat{P}'$ in a block $B'$ of $\widehat{M}'$. By Step~2(i), the surfaces $\widehat{S}'$ and $\widehat{P}'$ do not intersect in a basepoint-translate. Then the same is true for $S''$ and $P''$.

This concludes the proof of the main theorem except for the case where the core
is empty, which we shall now discuss. First, applying Corollary~\ref{cor:gamma does not backtrack} we pass to a
cover $M$, where the path $\gamma$ representing $g$ does not go through the same block twice.
By possibly passing to a further cover we also assume that $M$ is $S$--injective. Then instead
of the core we consider the minimal connected graph submanifold of $\widetilde{M}$ crossed by
both $\widetilde{S}$ and $\widetilde{P}$. Its blocks are in correspondence with some of the
blocks of $M$ crossed by $\gamma$. Steps~1 and~2 are now immediate. The surface $\widehat{S}'$ is
contained in a single block of the semicover $\widehat{M}'$. We extend $\widehat{M}'$ to a cover
$M'$ that is $S'$--injective and such that $S'\cap \widehat{M}'=\widehat{S}'$. We finally perform Step~3 as
in the main argument.
\end{proof}

\begin{rem}
\label{rem:annulus}
Theorems~\ref{thm:separability}~and~\ref{thm:double_separability} also hold
when $S$ and $P$ are allowed to be $\partial$--parallel annuli.
Theorem~\ref{thm:separability} follows directly from Corollary~\ref{cor:vertical_annuli}.

For Theorem~\ref{thm:double_separability}, if $P$ or $S$ are $\partial$--parallel
annuli, we homotope them into the boundary before we determine if their elevations $\widetilde{S},\widetilde{P}^0$ intersect. Without loss of generality we can assume that
$S$ is a $\partial$--parallel annulus, parallel to a boundary torus $T$. We identify $\mathrm{Stab}(\widetilde{S}), \mathrm{Stab}(\widetilde{P}^0)$ with $\pi_1S,\pi_1P$ for an appropriate basepoint.

If $P$ is also a $\partial$--parallel annulus, then it is also parallel to $T$, and it
suffices to use Corollary~\ref{cor:vertical separable} for separability of the finite index subgroup $\pi_1S\ \pi_1P$ of $\pi_1T$ in $\pi_1M$.

If $P$ is not a $\partial$--parallel annulus,
we can assume that $\pi_1S$ is not contained in
$\pi_1P\cap \pi_1T$. Let $H\subset \pi_1T$
be the finite index subgroup generated by $\pi_1S$ and $\pi_1P\cap\pi_1T$.
Then $H\pi_1P=\pi_1S\ \pi_1P$.
By separability of $H$ in $\pi_1M$ (Corollary~\ref{cor:vertical separable}), there is a finite cover $M'$ of $M$ with boundary torus $T'$
with fundamental group $H$. By Theorem~\ref{thm:double_separability} applied to $T'$ and
appropriate elevation $P'$ of $P$ to $M'$, we have that $H\pi_1P'$ is separable in $\pi_1M'$.
Hence $H\pi_1P$ is separable in $\pi_1M$ as desired.
\end{rem}

\end{document}